\documentclass[11pt]{amsart}
\usepackage{amssymb, graphicx, amsthm}
\usepackage{epsfig}
\usepackage{verbatim}
\usepackage{amsfonts}
\usepackage{mathrsfs}
\usepackage{amsbsy}
\usepackage{graphicx}
\renewcommand{\eqref}[1]{(\ref{#1})}

\newtheorem{prop}{Proposition}[section]
\newtheorem{lem}[prop]{Lemma}

\newtheorem{thm}{Theorem}[section]

\newtheorem{cor}[prop]{Corollary}

\usepackage{amscd}

\begin{document}

\title[Peudo-Anosov maps and simple closed geodesics]{Pseudo-Anosov maps and pairs of filling simple closed geodesics on Riemann surfaces  }
\author[C. Zhang]{C. Zhang}
\date{April 26, 2011}
\thanks{ }

\address{Department of Mathematics \\ Morehouse College
\\ Atlanta, GA 30314, USA.}
\email{czhang@morehouse.edu}


\maketitle 

\begin{abstract}
Let $S$ be a Riemann surface with a puncture $x$. Let $a\subset S$ be a simple closed geodesic. 
In this paper, we show that for any pseudo-Anosov map $f$ of $S$ that is isotopic to the identity on $S\cup \{x\}$, $(a, f^m(a))$ fills $S$ for $m\geq 3$. We also study the cases of $0<m\leq 2$ and show that if $(a,f^2(a))$ does not fill $S$, then there is only one geodesic $b$ such that $b$ is disjoint from both $a$ and $f^2(a)$. In fact, $b=f(a)$ and $\{a,f(a)\}$ forms the boundary of an $x$-punctured cylinder on $S$. As a consequence, we show that if $a$ and $f(a)$ are not disjoint. Then $(a,f^m(a))$ for any $m\geq 2$ fills $S$. 
\end{abstract}

\section{Introduction}
\setcounter{equation}{0} 

In an important paper \cite{Th}, Thurston proved that there exist pseudo-Anosov maps on a hyperbolic surface $S$ that are obtained from products of Dehn twists along two simple closed geodesics. More specifically, let $a,b$ be simple closed geodesics. The pair $(a,b)$ is called to fill $S$ if the union $a\cup b$ intersects every non-trivial closed curve, which is equivalent to that every component of $S\backslash \{a,b\}$ is a disk or an once punctured disk. Let $t_a, t_b$ denote the positive Dehn twists along $a$ and $b$, respectively. By \cite{Th} we know that $t_{a}\circ t_b^{-1}$ represents a pseudo-Anosov mapping class whenever $(a,b)$ fills $S$. See \cite{Th} and \cite{FLP} for a detailed account of pseudo-Anosov maps.  

Let $S$ be an analytically finite Riemann surface with type $(p,n)$, where $p$ is the genus and $n$ is the number of punctures of $S$. Assume that $3p-3+n>0$. Let $f:S\rightarrow S$ be a pseudo-Anosov map and let $a\subset S$ be a simple closed geodesic. Denote by $f^m(a)$ the geodesic homotopic to the image curve of $a$ under the map $f^m$. It is well-known \cite{FLP} that 
$$
\mathscr{S}=\{f^m(a):\ m\in \mathbf{Z}\}
$$ 
fills $S$ (in the sense that $S\backslash \mathscr{S}$ consists of disks or once punctured disks). Later, Fathi \cite{Fa} showed that a finite subset of $\mathscr{S}$ fills $S$. It is natural to ask if any pair of elements of $\mathscr{S}$ also fills $S$. Unfortunately, the answer to this question is ``no". In fact,  Wang--Wu--Zhou \cite{W-W-Z} showed that for any two non-separating non-isotopic simple closed geodesics $a,b$ on $S$, there is a pseudo-Anosov map $f$ such that $f(a)=b$. 

By contrast, in \cite{M-M}, Masur--Minsky showed that there is an integer $K$ such that  $(a,f^m(a))$ fills $S$ for all integers $m\geq K$. To determine the smallest possible integer $m$ with this property, Farb--Leininger--Margalit \cite{FLM} considered the curve complex $\mathscr{C}=\mathscr{C}(S)$ on $S$ that is equipped with the path metric $d_{\mathscr{C}}$, then they introduced the asymptotic translation length $\mu$ on $\mathscr{C}$ defined by $\mu=\lim_{m\rightarrow \infty} \inf d_{\mathscr{C}}(a,f^m(a))/m$, which is independent of choice of $a$, and showed that if $m$ is the smallest integer so that $m \mu > 2$, then $(a,f^{m}(a))$ fills $S$.

In this paper, we study the similar problem on a surface $S$ that contains at least one puncture $x$. Set $\tilde{S}=S\cup \{x\}$. Let $\mathscr{F}$ be the set of pseudo-Anosov maps of $S$ that are isotopic to the identity on $\tilde{S}$. Kra \cite{Kr} proved that $\mathscr{F}$ is non-empty and contains infinitely many elements. Let $f\in \mathscr{F}$ and let $F:[0,1]\times \tilde{S}\rightarrow \tilde{S}$ denote the isotopy between $f$ and the identity as $x$ is filled in. Then $\tilde{c}=F(t,x)$, $t\in [0,1]$, is a filling closed curve passing through $x$ in the sense that $\tilde{c}$ intersects every simple closed geodesic. 

For a simple closed geodesic $a$, we denote by $\tilde{a}$ the simple closed geodesic on $\tilde{S}$ homotopic to $a$ when $a$ is viewed as a curve on $\tilde{S}$. Let $K$ be the smallest integer such that $(a,f^m(a))$ fills $S$ whenever $m\geq K$. 
\begin{thm}\label{T1}
For an element $f\in \mathscr{F}$ and a simple closed geodesic $a\subset S$, we have $K\leq 3$. If $\tilde{a}$ is non-trivial and intersects $\tilde{c}$ more than once, then $K\leq 2$. If $\tilde{a}$ is trivial, then $K=1$. 
\end{thm}
It is well understood that every filling closed geodesic $\tilde{c}$ determines a conjugacy class of pseudo-Anosov maps on $S$ and one in the class, denoted by $f_{\tilde{c}}$, is constructed by pushing the point $x$ along $\tilde{c}$ until returning to its original position. 

Assume that $\tilde{c}$ is so chosen that there is an $x$-punctured cylinder $P\subset S$ such that $\tilde{c}$ goes through $P$ only once. Let $\partial P=\{a,a_0\}$ be the boundary components of $P$. In this case, the map $f_{\tilde{c}}$ satisfies the condition that $f_{\tilde{c}}(a)=a_0$. In particular, $f_{\tilde{c}}(a)$ and $a$ are disjoint. Since $f_{\tilde{c}}$ is a homeomorphism of $S$, $f_{\tilde{c}}(a)$ and $f_{\tilde{c}}^2(a)$ are also disjoint. We see that $f_{\tilde{c}}(a)$ is disjoint from both $a$ and $f_{\tilde{c}}^2(a)$. It follows that $(a,f_{\tilde{c}}^2(a))$ does not fill $S$. Our next result states that this is the only incidence for the pair $(a,f^2(a))$ not to fill $S$. More precisely, we
will prove the following result. 
\begin{thm}\label{T2}
Let $f\in \mathscr{F}$ and let  $a$ be a simple closed geodesic. Assume that $(a,f^2(a))$ does not fill $S$. Then there is a unique simple closed geodesic $b$ that is disjoint from both $a$ and $f^2(a)$. Furthermore, $b=f(a)$ and $\{a,b\}$ forms the boundary of an $x$-punctured cylinder on $S$. 
\end{thm}
An immediate consequence of Theorem \ref{T2} is the following corollary.
\begin{cor}\label{C1}
Let $a$, $f$ be given as in Theorem $1.2$. Assume that $a$ and $f(a)$ are not disjoint. Then $(a,f^m(a))$ for any $m\geq 2$ fills $S$. 
\end{cor}
This paper is organized as follows. In Section 2, we discuss properties of Dehn twists and their lifts to the universal covering space via the Bers isomorphism \cite{Bers1}. In Section 3, we study the actions of pseudo-Anosov maps in $\mathscr{F}$ on the set of simple closed geodesics on $S$. As an outcome, we obtain a family of pseudo-Anosov maps of $S$ determined by simple closed geodesics. In Section 4, we prove Theorem 1.1. In Section 5, we prove Theorem 1.2 and Corollary 1.1. Finally, we include some remarks and questions in Section 6. 
 
\section{Preliminaries}
\setcounter{equation}{0}

Let $\mathbf{D}=\{z\in \mathbf{C}:|z|<1\}$ be the hyperbolic disk with the Poincare metric $\rho(z)dz=dz/\left(1-|z|^2\right)$. Let $\varrho:\mathbf{D}\rightarrow \tilde{S}$ be the universal covering map with the covering group $G$. Let $\mbox{QC}(G)$ be the group of quasiconformal self-maps $w$ of $\mathbf{D}$ such that $wGw^{-1}=G$. Two maps $w,w'\in \mbox{QC}(G)$ are said equivalent if they share the common boundary values on $\partial \mathbf{D}=\mathbf{S}^1$. Denote by $[w]$ the equivalence class of an element $w\in \mbox{QC}(G)$. An important theorem of Bers \cite{Bers1} states that there is an isomorphism $\varphi^*$ of the quotient group  $\mbox{QC}(G)/\!\sim$ onto the $x$-pointed mapping class group Mod$_S^x$. In what follows, $\varphi^*:\mbox{QC}(G)/\! \sim \rightarrow \mbox{Mod}_S^x$ is called the Bers isomorphism. Under the Bers isomorphism, the image $\varphi^*(G)$ is a subgroup of Mod$_S^x$ consisting of elements that project to a trivial mapping class of $\tilde{S}=\mathbf{D}/G$. By abuse of language, we use $[w]^*$ to denote the mapping class $\varphi^*([w])$ as well as a representative of $\varphi^*([w])$ for an element $w\in \mbox{QC}(G)$. In particular, for an element $h\in G$, we use $h^*$ to denote the mapping class  $\varphi^*(h)$ as well as a representative of $\varphi^*(h)$. 

Let $\tilde{c}\subset \tilde{S}$ be a closed geodesic. Note that any geodesic $\hat{c}\subset \mathbf{D}$ with $\varrho(\hat{c})=\tilde{c}$ is an invariant geodesic under a hyperbolic element $g_{\hat{c}}\in G$. Following Kra \cite{Kr}, $g_{\hat{c}}$ is called an essential hyperbolic element if $g_{\hat{c}}$ corresponds to an element $g_{\hat{c}}^*=f_{\hat{c}}$ in $\mathscr{F}$. In this case, $\tilde{c}=\varrho(\hat{c})$ is a filling geodesic. 

Consider now some special elements in $\mbox{QC}(G)/\!\sim$. Let $a\subset S$ be a simple closed geodesic that is non-trivial on $\tilde{S}$ as $x$ is filled in. Let $\tilde{a}$ denote the (non-trivial) simple closed geodesic homotopic to $a$ on $\tilde{S}$. We can construct a Dehn twist along $\tilde{a}$ as follows. We first cut $\tilde{S}$ along $\tilde{a}$, rotate one of the copies of $\tilde{a}$ by 360 degrees in the counterclockwise direction and then glue the two copies back together.  

Let $\hat{a}\subset \mathbf{D}$ be a geodesic such that $\varrho(\hat{a})=\tilde{a}$. Denote by $\{U, U'\}$ the components of $\mathbf{D}\backslash \{\hat{a}\}$. It is readily seen that $\hat{a}$, $U$ and $U'$ are invariant by a primitive simple hyperbolic element of $G$. The Dehn twist $t_{\tilde{a}}$ can be lifted to a map $\tau_{\hat{a}}:\mathbf{D}\rightarrow \mathbf{D}$ with respect to $U$, which satisfies the conditions: 
$$
\mbox{(i)} \ \ \tau_{\hat{a}}G\tau_{\hat{a}}^{-1}=G \ \ \ \mbox{and}\ \ \  \mbox{(ii)} \ \ \varrho\circ \tau_{\hat{a}}=t_{\tilde{a}}\circ \varrho.
$$
In addition to (i) and (ii) above, $\tau_{\hat{a}}$ defines a collection $\mathscr{U}_{\hat{a}}$ of half planes of $\mathbf{D}$ in a partial order defined by inclusion, and all maximal elements $U_i$ ($U$ is one of them) of $\mathscr{U}_{\hat{a}}$ are mutually disjoint, and the complement 
\begin{equation}\label{LLMM}
\Omega_{\hat{a}}=\mathbf{D}\backslash \bigcup_i U_i\subset U'
\end{equation}
\noindent is not empty and is a convex region bounded by a collection of disjoint geodesics $\hat{a}$ with $\varrho(\hat{a})=\tilde{a}$. Clearly, $U'$ contains infinitely many maximal elements of $\mathscr{U}_{\hat{a}}$. The map $\tau_{\hat{a}}$ keeps each maximal element invariant, and restricts to the identity on $\Omega_{\hat{a}}$.  

We remark that $\tau_{\hat{a}}$ so obtained depends on the choice of a geodesic $\hat{a}$ with $\varrho(\hat{a})=\tilde{a}$, but
does not depend on the choice of a boundary component of $\Omega_{\hat{a}}$. Note also that all lifts of $t_{\tilde{a}}$ are in the forms $h\circ \tau_{\hat{a}}$ for $h\in G$. Furthermore, $\tau_{\hat{a}}$ determines an element $[\tau_{\hat{a}}]$ of QC$(G)/\! \sim$. By Lemma 3.2 of \cite{CZ1}, we can choose $\hat{a}$ (and hence $U$) so that $[\tau_{\hat{a}}]^*\in \mbox{Mod}_S^x$ is represented by the Dehn twist  $t_a$ along the geodesic $a$. See \cite{CZ1,CZ2} for more details. The following lemma, deduced from the definition of $\tau_{\hat{a}}$, plays an important role in this paper. 
\begin{lem}\label{A1}
Let $U\in \mathscr{U}_{\hat{a}}$ be a maximal element. Let $h\in G$ be a hyperbolic element whose axis  $c_h$ crosses the boundary $\partial U$ of $U$. Assume that $U$ covers the repelling fixed point of $h$. Then $h(\mathbf{D}\backslash U)$ is contained in another maximal element $U_0\in \mathscr{U}_{\hat{\alpha}}$. Moreover, $h(\mathbf{D}\backslash U)=U_0$ if and only if the geodesic $\varrho(c_h)$ intersects $\tilde{a}$ exactly once. 
\end{lem}
\begin{proof}
Under the universal covering map $\varrho:\mathbf{D}\rightarrow \tilde{S}$, $\partial U$ projects to $\tilde{a}$. Parametrize $c_h=c_h(t)$, $-\infty<t<+\infty$, so that $c_h(-\infty)$ is the repelling fixed point of $h$ while $c_h(+\infty)$ is the attracting fixed point of $h$. Let $c_h(t_0)=c_h\cap \partial U$. 
Let $\tilde{\gamma}$ denote the projection of $c_h$ under the map $\varrho$. Then $\tilde{\gamma}$ is a non-trivial closed geodesic. 

Now $\varrho(c_h(t_0))$ is one of the points of intersection between $\tilde{a}$ and $\tilde{\gamma}$. Since both $\tilde{a}$ and $\tilde{\gamma}$ are closed, there must exist a number $T>0$ such that $\varrho(c_h(t_0+T))$ is also an intersection point of $\tilde{a}$ and $\tilde{\gamma}$. This implies that there is an element $U'\in \mathscr{U}_{\hat{a}}$, disjoint from $U$, such tat $c_h(t_0+T)=\partial U'\cap c_h$. But $U'$ must be included in a maximal element $U_0$ of $\mathscr{U}_{\hat{a}}$. Clearly, $U_0$ is disjoint from $U$,  and $U_0=U'$ if and only if $\varrho(c_h)$ intersects $\tilde{a}$ exactly once. 
\end{proof}

\section{Simple closed geodesics under the actions of elements of $\mathscr{F}$}
\setcounter{equation}{0}

In this section, we study the actions of essential hyperbolic elements $g$ of $G$ on the configuration determined by $\mathscr{U}_{\hat{a}}$. The results lead to new constructions of pseudo-Anosov maps on $S$. 
 
Let $f\in \mathscr{F}$. Write $f=g^*$ for some $g\in G$. Then $g$ is an essential hyperbolic element whose axis $\hat{c}=\hat{c}_g$ projects to a filling closed geodesic $\tilde{c}$ on $\tilde{S}$ under the universal covering map $\varrho:\mathbf{D}\rightarrow \tilde{S}$. 

Let $a\subset S$ be a simple closed geodesic. If $a$ projects to a trivial curve on $\tilde{S}$, then by Theorem 1.1 of \cite{CZ3}, $(a,f^m(a))$ fills $S$ for all $m\geq 1$. Thus we may assume in this section that $a$ projects to a non-trivial curve that is homotopic to $\tilde{a}$ on $\tilde{S}$.

Assume that  $\hat{c}\cap \Omega_{\hat{a}}\neq \emptyset$. Since $\tilde{c}$ intersects $\tilde{a}$, $\hat{c}$ intersects a geodesic $\hat{a}$ for which $\varrho(\hat{a})=\tilde{a}$. Let $U\in \mathscr{U}_{\hat{a}}$ be a maximal element such that $\partial U=\hat{a}$. Assume also that $U$ covers the attracting fixed point of $g$ (by reversing the orientation if necessary). By Lemma \ref{A1}, there is another maximal element $U_0\in \mathscr{U}_{\hat{a}}$ that covers the repelling fixed point of $g$. Let $\hat{b}=\partial U_0$. See Figure 1 (a). In the figure both $\hat{a}$ and $\hat{b}$ project to $\tilde{a}$ under $\varrho:\mathbf{D}\rightarrow \tilde{S}$. Observe also that $g(\hat{b})$ does not intersect $\hat{a}$, but $g(\hat{b})$ could be equal to $\hat{a}$. 

\bigskip
\medskip

%TeXCAD Picture [figure002.pic]. Options:
%\grade{\on}
%\emlines{\off}
%\epic{\off}
%\beziermacro{\on}
%\reduce{\on}
%\snapping{\off}
%\pvinsert{% Your \input, \def, etc. here}
%\quality{8.000}
%\graddiff{0.005}
%\snapasp{1}
%\zoom{4.0000}
\unitlength 1mm % = 2.845pt
\linethickness{0.4pt}
\ifx\plotpoint\undefined\newsavebox{\plotpoint}\fi % GNUPLOT compatibility
\begin{picture}(131.5,59.959)(0,0)
%\circle(35.75,34.5){50.917}
\put(61.209,34.5){\line(0,1){1.0884}}
\put(61.185,35.588){\line(0,1){1.0865}}
\put(61.115,36.675){\line(0,1){1.0825}}
\multiput(60.999,37.757)(-.032471,.215305){5}{\line(0,1){.215305}}
\multiput(60.837,38.834)(-.029747,.152657){7}{\line(0,1){.152657}}
\multiput(60.629,39.903)(-.031716,.13234){8}{\line(0,1){.13234}}
\multiput(60.375,40.961)(-.033195,.116323){9}{\line(0,1){.116323}}
\multiput(60.076,42.008)(-.031204,.093925){11}{\line(0,1){.093925}}
\multiput(59.733,43.041)(-.032259,.084796){12}{\line(0,1){.084796}}
\multiput(59.346,44.059)(-.0330964,.0769287){13}{\line(0,1){.0769287}}
\multiput(58.916,45.059)(-.0315078,.0653842){15}{\line(0,1){.0653842}}
\multiput(58.443,46.04)(-.0321323,.0599788){16}{\line(0,1){.0599788}}
\multiput(57.929,46.999)(-.032628,.055106){17}{\line(0,1){.055106}}
\multiput(57.374,47.936)(-.0330122,.0506795){18}{\line(0,1){.0506795}}
\multiput(56.78,48.848)(-.0332989,.0466312){19}{\line(0,1){.0466312}}
\multiput(56.147,49.734)(-.033499,.0429066){20}{\line(0,1){.0429066}}
\multiput(55.477,50.593)(-.0336217,.0394621){21}{\line(0,1){.0394621}}
\multiput(54.771,51.421)(-.0336745,.0362618){22}{\line(0,1){.0362618}}
\multiput(54.03,52.219)(-.0336639,.0332764){23}{\line(-1,0){.0336639}}
\multiput(53.256,52.984)(-.0366493,.0332524){22}{\line(-1,0){.0366493}}
\multiput(52.45,53.716)(-.0398487,.0331625){21}{\line(-1,0){.0398487}}
\multiput(51.613,54.412)(-.0432916,.0329999){20}{\line(-1,0){.0432916}}
\multiput(50.747,55.072)(-.0470136,.0327567){19}{\line(-1,0){.0470136}}
\multiput(49.854,55.695)(-.0510584,.0324232){18}{\line(-1,0){.0510584}}
\multiput(48.935,56.278)(-.0554801,.0319877){17}{\line(-1,0){.0554801}}
\multiput(47.992,56.822)(-.0643699,.0335313){15}{\line(-1,0){.0643699}}
\multiput(47.026,57.325)(-.0704407,.032945){14}{\line(-1,0){.0704407}}
\multiput(46.04,57.786)(-.0773067,.0322035){13}{\line(-1,0){.0773067}}
\multiput(45.035,58.205)(-.085164,.031275){12}{\line(-1,0){.085164}}
\multiput(44.013,58.58)(-.103708,.033126){10}{\line(-1,0){.103708}}
\multiput(42.976,58.911)(-.116699,.031846){9}{\line(-1,0){.116699}}
\multiput(41.926,59.198)(-.132698,.030181){8}{\line(-1,0){.132698}}
\multiput(40.864,59.44)(-.17849,.032641){6}{\line(-1,0){.17849}}
\multiput(39.793,59.635)(-.215666,.029975){5}{\line(-1,0){.215666}}
\put(38.715,59.785){\line(-1,0){1.0838}}
\put(37.631,59.889){\line(-1,0){1.0872}}
\put(36.544,59.946){\line(-1,0){1.0886}}
\put(35.455,59.957){\line(-1,0){1.0881}}
\put(34.367,59.921){\line(-1,0){1.0856}}
\multiput(33.282,59.839)(-.27027,-.03218){4}{\line(-1,0){.27027}}
\multiput(32.2,59.71)(-.179095,-.029135){6}{\line(-1,0){.179095}}
\multiput(31.126,59.535)(-.152302,-.031513){7}{\line(-1,0){.152302}}
\multiput(30.06,59.314)(-.131964,-.033246){8}{\line(-1,0){.131964}}
\multiput(29.004,59.048)(-.104337,-.031086){10}{\line(-1,0){.104337}}
\multiput(27.961,58.738)(-.093557,-.03229){11}{\line(-1,0){.093557}}
\multiput(26.932,58.382)(-.084417,-.033238){12}{\line(-1,0){.084417}}
\multiput(25.919,57.984)(-.0710731,-.0315575){14}{\line(-1,0){.0710731}}
\multiput(24.924,57.542)(-.065015,-.0322628){15}{\line(-1,0){.065015}}
\multiput(23.948,57.058)(-.0596027,-.0328246){16}{\line(-1,0){.0596027}}
\multiput(22.995,56.533)(-.0547245,-.0332639){17}{\line(-1,0){.0547245}}
\multiput(22.064,55.967)(-.0502939,-.0335968){18}{\line(-1,0){.0502939}}
\multiput(21.159,55.362)(-.0439303,-.0321447){20}{\line(-1,0){.0439303}}
\multiput(20.28,54.72)(-.0404913,-.0323748){21}{\line(-1,0){.0404913}}
\multiput(19.43,54.04)(-.0372942,-.0325274){22}{\line(-1,0){.0372942}}
\multiput(18.61,53.324)(-.0343099,-.0326099){23}{\line(-1,0){.0343099}}
\multiput(17.821,52.574)(-.0328843,-.034047){23}{\line(0,-1){.034047}}
\multiput(17.064,51.791)(-.0328258,-.0370318){22}{\line(0,-1){.0370318}}
\multiput(16.342,50.976)(-.0326989,-.04023){21}{\line(0,-1){.04023}}
\multiput(15.655,50.131)(-.0324964,-.0436708){20}{\line(0,-1){.0436708}}
\multiput(15.005,49.258)(-.0322101,-.0473897){19}{\line(0,-1){.0473897}}
\multiput(14.393,48.358)(-.0337022,-.0544557){17}{\line(0,-1){.0544557}}
\multiput(13.82,47.432)(-.0333021,-.0593372){16}{\line(0,-1){.0593372}}
\multiput(13.288,46.482)(-.0327838,-.0647539){15}{\line(0,-1){.0647539}}
\multiput(12.796,45.511)(-.0321271,-.0708175){14}{\line(0,-1){.0708175}}
\multiput(12.346,44.52)(-.0313062,-.0776744){13}{\line(0,-1){.0776744}}
\multiput(11.939,43.51)(-.03304,-.093295){11}{\line(0,-1){.093295}}
\multiput(11.576,42.484)(-.031923,-.104084){10}{\line(0,-1){.104084}}
\multiput(11.256,41.443)(-.030493,-.11706){9}{\line(0,-1){.11706}}
\multiput(10.982,40.389)(-.032735,-.152044){7}{\line(0,-1){.152044}}
\multiput(10.753,39.325)(-.030572,-.178856){6}{\line(0,-1){.178856}}
\multiput(10.569,38.252)(-.027476,-.215999){5}{\line(0,-1){.215999}}
\put(10.432,37.172){\line(0,-1){1.0849}}
\put(10.341,36.087){\line(0,-1){4.3487}}
\multiput(10.442,31.738)(.028243,-.2159){5}{\line(0,-1){.2159}}
\multiput(10.583,30.659)(.031206,-.178746){6}{\line(0,-1){.178746}}
\multiput(10.77,29.586)(.033274,-.151927){7}{\line(0,-1){.151927}}
\multiput(11.003,28.523)(.030908,-.116951){9}{\line(0,-1){.116951}}
\multiput(11.281,27.47)(.032292,-.10397){10}{\line(0,-1){.10397}}
\multiput(11.604,26.431)(.033371,-.093177){11}{\line(0,-1){.093177}}
\multiput(11.971,25.406)(.0315817,-.0775628){13}{\line(0,-1){.0775628}}
\multiput(12.382,24.397)(.0323783,-.070703){14}{\line(0,-1){.070703}}
\multiput(12.835,23.407)(.0330134,-.0646371){15}{\line(0,-1){.0646371}}
\multiput(13.33,22.438)(.0335125,-.0592186){16}{\line(0,-1){.0592186}}
\multiput(13.867,21.49)(.0320122,-.0513171){18}{\line(0,-1){.0513171}}
\multiput(14.443,20.567)(.0323781,-.0472751){19}{\line(0,-1){.0472751}}
\multiput(15.058,19.668)(.0326512,-.0435552){20}{\line(0,-1){.0435552}}
\multiput(15.711,18.797)(.0328415,-.0401137){21}{\line(0,-1){.0401137}}
\multiput(16.401,17.955)(.0329571,-.0369151){22}{\line(0,-1){.0369151}}
\multiput(17.126,17.143)(.033005,-.03393){23}{\line(0,-1){.03393}}
\multiput(17.885,16.362)(.0344254,-.0324879){23}{\line(1,0){.0344254}}
\multiput(18.677,15.615)(.0374094,-.0323948){22}{\line(1,0){.0374094}}
\multiput(19.5,14.903)(.040606,-.0322309){21}{\line(1,0){.040606}}
\multiput(20.352,14.226)(.0463623,-.0336722){19}{\line(1,0){.0463623}}
\multiput(21.233,13.586)(.0504128,-.0334181){18}{\line(1,0){.0504128}}
\multiput(22.141,12.984)(.0548423,-.0330694){17}{\line(1,0){.0548423}}
\multiput(23.073,12.422)(.0597188,-.0326128){16}{\line(1,0){.0597188}}
\multiput(24.028,11.9)(.0651291,-.0320318){15}{\line(1,0){.0651291}}
\multiput(25.005,11.42)(.0766605,-.0337131){13}{\line(1,0){.0766605}}
\multiput(26.002,10.982)(.084534,-.032938){12}{\line(1,0){.084534}}
\multiput(27.016,10.586)(.093671,-.031957){11}{\line(1,0){.093671}}
\multiput(28.047,10.235)(.104447,-.030716){10}{\line(1,0){.104447}}
\multiput(29.091,9.928)(.132081,-.032777){8}{\line(1,0){.132081}}
\multiput(30.148,9.666)(.152413,-.030972){7}{\line(1,0){.152413}}
\multiput(31.215,9.449)(.179198,-.028499){6}{\line(1,0){.179198}}
\multiput(32.29,9.278)(.27038,-.03122){4}{\line(1,0){.27038}}
\put(33.371,9.153){\line(1,0){1.0859}}
\put(34.457,9.074){\line(1,0){1.0882}}
\put(35.546,9.042){\line(1,0){1.0886}}
\put(36.634,9.057){\line(1,0){1.087}}
\put(37.721,9.118){\line(1,0){1.0834}}
\multiput(38.805,9.225)(.215559,.030741){5}{\line(1,0){.215559}}
\multiput(39.882,9.379)(.178373,.033274){6}{\line(1,0){.178373}}
\multiput(40.953,9.579)(.13259,.030652){8}{\line(1,0){.13259}}
\multiput(42.013,9.824)(.116585,.03226){9}{\line(1,0){.116585}}
\multiput(43.063,10.114)(.10359,.033494){10}{\line(1,0){.10359}}
\multiput(44.098,10.449)(.085052,.031577){12}{\line(1,0){.085052}}
\multiput(45.119,10.828)(.0771919,.0324777){13}{\line(1,0){.0771919}}
\multiput(46.123,11.25)(.0703233,.0331948){14}{\line(1,0){.0703233}}
\multiput(47.107,11.715)(.0602348,.0316496){16}{\line(1,0){.0602348}}
\multiput(48.071,12.222)(.0553662,.0321845){17}{\line(1,0){.0553662}}
\multiput(49.012,12.769)(.050943,.0326042){18}{\line(1,0){.050943}}
\multiput(49.929,13.356)(.046897,.0329234){19}{\line(1,0){.046897}}
\multiput(50.82,13.981)(.0431742,.0331534){20}{\line(1,0){.0431742}}
\multiput(51.684,14.644)(.0397308,.0333037){21}{\line(1,0){.0397308}}
\multiput(52.518,15.344)(.036531,.0333823){22}{\line(1,0){.036531}}
\multiput(53.322,16.078)(.0335456,.0333956){23}{\line(1,0){.0335456}}
\multiput(54.093,16.846)(.0335456,.0363811){22}{\line(0,1){.0363811}}
\multiput(54.831,17.646)(.0334814,.0395812){21}{\line(0,1){.0395812}}
\multiput(55.534,18.478)(.0333465,.0430253){20}{\line(0,1){.0430253}}
\multiput(56.201,19.338)(.0331331,.0467491){19}{\line(0,1){.0467491}}
\multiput(56.831,20.226)(.0328321,.0507964){18}{\line(0,1){.0507964}}
\multiput(57.422,21.141)(.0324322,.0552215){17}{\line(0,1){.0552215}}
\multiput(57.973,22.079)(.0319192,.0600925){16}{\line(0,1){.0600925}}
\multiput(58.484,23.041)(.0335095,.0701739){14}{\line(0,1){.0701739}}
\multiput(58.953,24.023)(.0328232,.0770457){13}{\line(0,1){.0770457}}
\multiput(59.38,25.025)(.031957,.08491){12}{\line(0,1){.08491}}
\multiput(59.763,26.044)(.03087,.094035){11}{\line(0,1){.094035}}
\multiput(60.103,27.078)(.032782,.11644){9}{\line(0,1){.11644}}
\multiput(60.398,28.126)(.031246,.132452){8}{\line(0,1){.132452}}
\multiput(60.648,29.186)(.029205,.152762){7}{\line(0,1){.152762}}
\multiput(60.852,30.255)(.031706,.215419){5}{\line(0,1){.215419}}
\put(61.011,31.332){\line(0,1){1.0829}}
\put(61.123,32.415){\line(0,1){2.0849}}
%\end
%\circle(103.25,34.5){50.917}
\put(128.708,34.5){\line(0,1){1.0884}}
\put(128.685,35.588){\line(0,1){1.0865}}
\put(128.615,36.675){\line(0,1){1.0825}}
\multiput(128.499,37.757)(-.032471,.215305){5}{\line(0,1){.215305}}
\multiput(128.337,38.834)(-.029747,.152657){7}{\line(0,1){.152657}}
\multiput(128.129,39.903)(-.031716,.13234){8}{\line(0,1){.13234}}
\multiput(127.875,40.961)(-.033195,.116323){9}{\line(0,1){.116323}}
\multiput(127.576,42.008)(-.031204,.093925){11}{\line(0,1){.093925}}
\multiput(127.233,43.041)(-.032259,.084796){12}{\line(0,1){.084796}}
\multiput(126.846,44.059)(-.0330964,.0769287){13}{\line(0,1){.0769287}}
\multiput(126.416,45.059)(-.0315078,.0653842){15}{\line(0,1){.0653842}}
\multiput(125.943,46.04)(-.0321323,.0599788){16}{\line(0,1){.0599788}}
\multiput(125.429,46.999)(-.032628,.055106){17}{\line(0,1){.055106}}
\multiput(124.874,47.936)(-.0330122,.0506795){18}{\line(0,1){.0506795}}
\multiput(124.28,48.848)(-.0332989,.0466312){19}{\line(0,1){.0466312}}
\multiput(123.647,49.734)(-.033499,.0429066){20}{\line(0,1){.0429066}}
\multiput(122.977,50.593)(-.0336217,.0394621){21}{\line(0,1){.0394621}}
\multiput(122.271,51.421)(-.0336745,.0362618){22}{\line(0,1){.0362618}}
\multiput(121.53,52.219)(-.0336639,.0332764){23}{\line(-1,0){.0336639}}
\multiput(120.756,52.984)(-.0366493,.0332524){22}{\line(-1,0){.0366493}}
\multiput(119.95,53.716)(-.0398487,.0331625){21}{\line(-1,0){.0398487}}
\multiput(119.113,54.412)(-.0432916,.0329999){20}{\line(-1,0){.0432916}}
\multiput(118.247,55.072)(-.0470136,.0327567){19}{\line(-1,0){.0470136}}
\multiput(117.354,55.695)(-.0510584,.0324232){18}{\line(-1,0){.0510584}}
\multiput(116.435,56.278)(-.0554801,.0319877){17}{\line(-1,0){.0554801}}
\multiput(115.492,56.822)(-.0643699,.0335313){15}{\line(-1,0){.0643699}}
\multiput(114.526,57.325)(-.0704407,.032945){14}{\line(-1,0){.0704407}}
\multiput(113.54,57.786)(-.0773067,.0322035){13}{\line(-1,0){.0773067}}
\multiput(112.535,58.205)(-.085164,.031275){12}{\line(-1,0){.085164}}
\multiput(111.513,58.58)(-.103708,.033126){10}{\line(-1,0){.103708}}
\multiput(110.476,58.911)(-.116699,.031846){9}{\line(-1,0){.116699}}
\multiput(109.426,59.198)(-.132698,.030181){8}{\line(-1,0){.132698}}
\multiput(108.364,59.44)(-.17849,.032641){6}{\line(-1,0){.17849}}
\multiput(107.293,59.635)(-.215666,.029975){5}{\line(-1,0){.215666}}
\put(106.215,59.785){\line(-1,0){1.0838}}
\put(105.131,59.889){\line(-1,0){1.0872}}
\put(104.044,59.946){\line(-1,0){1.0886}}
\put(102.955,59.957){\line(-1,0){1.0881}}
\put(101.867,59.921){\line(-1,0){1.0856}}
\multiput(100.782,59.839)(-.27027,-.03218){4}{\line(-1,0){.27027}}
\multiput(99.7,59.71)(-.179095,-.029135){6}{\line(-1,0){.179095}}
\multiput(98.626,59.535)(-.152302,-.031513){7}{\line(-1,0){.152302}}
\multiput(97.56,59.314)(-.131964,-.033246){8}{\line(-1,0){.131964}}
\multiput(96.504,59.048)(-.104337,-.031086){10}{\line(-1,0){.104337}}
\multiput(95.461,58.738)(-.093557,-.03229){11}{\line(-1,0){.093557}}
\multiput(94.432,58.382)(-.084417,-.033238){12}{\line(-1,0){.084417}}
\multiput(93.419,57.984)(-.0710731,-.0315575){14}{\line(-1,0){.0710731}}
\multiput(92.424,57.542)(-.065015,-.0322628){15}{\line(-1,0){.065015}}
\multiput(91.448,57.058)(-.0596027,-.0328246){16}{\line(-1,0){.0596027}}
\multiput(90.495,56.533)(-.0547245,-.0332639){17}{\line(-1,0){.0547245}}
\multiput(89.564,55.967)(-.0502939,-.0335968){18}{\line(-1,0){.0502939}}
\multiput(88.659,55.362)(-.0439303,-.0321447){20}{\line(-1,0){.0439303}}
\multiput(87.78,54.72)(-.0404913,-.0323748){21}{\line(-1,0){.0404913}}
\multiput(86.93,54.04)(-.0372942,-.0325274){22}{\line(-1,0){.0372942}}
\multiput(86.11,53.324)(-.0343099,-.0326099){23}{\line(-1,0){.0343099}}
\multiput(85.321,52.574)(-.0328843,-.034047){23}{\line(0,-1){.034047}}
\multiput(84.564,51.791)(-.0328258,-.0370318){22}{\line(0,-1){.0370318}}
\multiput(83.842,50.976)(-.0326989,-.04023){21}{\line(0,-1){.04023}}
\multiput(83.155,50.131)(-.0324964,-.0436708){20}{\line(0,-1){.0436708}}
\multiput(82.505,49.258)(-.0322101,-.0473897){19}{\line(0,-1){.0473897}}
\multiput(81.893,48.358)(-.0337022,-.0544557){17}{\line(0,-1){.0544557}}
\multiput(81.32,47.432)(-.0333021,-.0593372){16}{\line(0,-1){.0593372}}
\multiput(80.788,46.482)(-.0327838,-.0647539){15}{\line(0,-1){.0647539}}
\multiput(80.296,45.511)(-.0321271,-.0708175){14}{\line(0,-1){.0708175}}
\multiput(79.846,44.52)(-.0313062,-.0776744){13}{\line(0,-1){.0776744}}
\multiput(79.439,43.51)(-.03304,-.093295){11}{\line(0,-1){.093295}}
\multiput(79.076,42.484)(-.031923,-.104084){10}{\line(0,-1){.104084}}
\multiput(78.756,41.443)(-.030493,-.11706){9}{\line(0,-1){.11706}}
\multiput(78.482,40.389)(-.032735,-.152044){7}{\line(0,-1){.152044}}
\multiput(78.253,39.325)(-.030572,-.178856){6}{\line(0,-1){.178856}}
\multiput(78.069,38.252)(-.027476,-.215999){5}{\line(0,-1){.215999}}
\put(77.932,37.172){\line(0,-1){1.0849}}
\put(77.841,36.087){\line(0,-1){4.3487}}
\multiput(77.942,31.738)(.028243,-.2159){5}{\line(0,-1){.2159}}
\multiput(78.083,30.659)(.031206,-.178746){6}{\line(0,-1){.178746}}
\multiput(78.27,29.586)(.033274,-.151927){7}{\line(0,-1){.151927}}
\multiput(78.503,28.523)(.030908,-.116951){9}{\line(0,-1){.116951}}
\multiput(78.781,27.47)(.032292,-.10397){10}{\line(0,-1){.10397}}
\multiput(79.104,26.431)(.033371,-.093177){11}{\line(0,-1){.093177}}
\multiput(79.471,25.406)(.0315817,-.0775628){13}{\line(0,-1){.0775628}}
\multiput(79.882,24.397)(.0323783,-.070703){14}{\line(0,-1){.070703}}
\multiput(80.335,23.407)(.0330134,-.0646371){15}{\line(0,-1){.0646371}}
\multiput(80.83,22.438)(.0335125,-.0592186){16}{\line(0,-1){.0592186}}
\multiput(81.367,21.49)(.0320122,-.0513171){18}{\line(0,-1){.0513171}}
\multiput(81.943,20.567)(.0323781,-.0472751){19}{\line(0,-1){.0472751}}
\multiput(82.558,19.668)(.0326512,-.0435552){20}{\line(0,-1){.0435552}}
\multiput(83.211,18.797)(.0328415,-.0401137){21}{\line(0,-1){.0401137}}
\multiput(83.901,17.955)(.0329571,-.0369151){22}{\line(0,-1){.0369151}}
\multiput(84.626,17.143)(.033005,-.03393){23}{\line(0,-1){.03393}}
\multiput(85.385,16.362)(.0344254,-.0324879){23}{\line(1,0){.0344254}}
\multiput(86.177,15.615)(.0374094,-.0323948){22}{\line(1,0){.0374094}}
\multiput(87,14.903)(.040606,-.0322309){21}{\line(1,0){.040606}}
\multiput(87.852,14.226)(.0463623,-.0336722){19}{\line(1,0){.0463623}}
\multiput(88.733,13.586)(.0504128,-.0334181){18}{\line(1,0){.0504128}}
\multiput(89.641,12.984)(.0548423,-.0330694){17}{\line(1,0){.0548423}}
\multiput(90.573,12.422)(.0597188,-.0326128){16}{\line(1,0){.0597188}}
\multiput(91.528,11.9)(.0651291,-.0320318){15}{\line(1,0){.0651291}}
\multiput(92.505,11.42)(.0766605,-.0337131){13}{\line(1,0){.0766605}}
\multiput(93.502,10.982)(.084534,-.032938){12}{\line(1,0){.084534}}
\multiput(94.516,10.586)(.093671,-.031957){11}{\line(1,0){.093671}}
\multiput(95.547,10.235)(.104447,-.030716){10}{\line(1,0){.104447}}
\multiput(96.591,9.928)(.132081,-.032777){8}{\line(1,0){.132081}}
\multiput(97.648,9.666)(.152413,-.030972){7}{\line(1,0){.152413}}
\multiput(98.715,9.449)(.179198,-.028499){6}{\line(1,0){.179198}}
\multiput(99.79,9.278)(.27038,-.03122){4}{\line(1,0){.27038}}
\put(100.871,9.153){\line(1,0){1.0859}}
\put(101.957,9.074){\line(1,0){1.0882}}
\put(103.046,9.042){\line(1,0){1.0886}}
\put(104.134,9.057){\line(1,0){1.087}}
\put(105.221,9.118){\line(1,0){1.0834}}
\multiput(106.305,9.225)(.215559,.030741){5}{\line(1,0){.215559}}
\multiput(107.382,9.379)(.178373,.033274){6}{\line(1,0){.178373}}
\multiput(108.453,9.579)(.13259,.030652){8}{\line(1,0){.13259}}
\multiput(109.513,9.824)(.116585,.03226){9}{\line(1,0){.116585}}
\multiput(110.563,10.114)(.10359,.033494){10}{\line(1,0){.10359}}
\multiput(111.598,10.449)(.085052,.031577){12}{\line(1,0){.085052}}
\multiput(112.619,10.828)(.0771919,.0324777){13}{\line(1,0){.0771919}}
\multiput(113.623,11.25)(.0703233,.0331948){14}{\line(1,0){.0703233}}
\multiput(114.607,11.715)(.0602348,.0316496){16}{\line(1,0){.0602348}}
\multiput(115.571,12.222)(.0553662,.0321845){17}{\line(1,0){.0553662}}
\multiput(116.512,12.769)(.050943,.0326042){18}{\line(1,0){.050943}}
\multiput(117.429,13.356)(.046897,.0329234){19}{\line(1,0){.046897}}
\multiput(118.32,13.981)(.0431742,.0331534){20}{\line(1,0){.0431742}}
\multiput(119.184,14.644)(.0397308,.0333037){21}{\line(1,0){.0397308}}
\multiput(120.018,15.344)(.036531,.0333823){22}{\line(1,0){.036531}}
\multiput(120.822,16.078)(.0335456,.0333956){23}{\line(1,0){.0335456}}
\multiput(121.593,16.846)(.0335456,.0363811){22}{\line(0,1){.0363811}}
\multiput(122.331,17.646)(.0334814,.0395812){21}{\line(0,1){.0395812}}
\multiput(123.034,18.478)(.0333465,.0430253){20}{\line(0,1){.0430253}}
\multiput(123.701,19.338)(.0331331,.0467491){19}{\line(0,1){.0467491}}
\multiput(124.331,20.226)(.0328321,.0507964){18}{\line(0,1){.0507964}}
\multiput(124.922,21.141)(.0324322,.0552215){17}{\line(0,1){.0552215}}
\multiput(125.473,22.079)(.0319192,.0600925){16}{\line(0,1){.0600925}}
\multiput(125.984,23.041)(.0335095,.0701739){14}{\line(0,1){.0701739}}
\multiput(126.453,24.023)(.0328232,.0770457){13}{\line(0,1){.0770457}}
\multiput(126.88,25.025)(.031957,.08491){12}{\line(0,1){.08491}}
\multiput(127.263,26.044)(.03087,.094035){11}{\line(0,1){.094035}}
\multiput(127.603,27.078)(.032782,.11644){9}{\line(0,1){.11644}}
\multiput(127.898,28.126)(.031246,.132452){8}{\line(0,1){.132452}}
\multiput(128.148,29.186)(.029205,.152762){7}{\line(0,1){.152762}}
\multiput(128.352,30.255)(.031706,.215419){5}{\line(0,1){.215419}}
\put(128.511,31.332){\line(0,1){1.0829}}
\put(128.623,32.415){\line(0,1){2.0849}}
%\end
\qbezier(15.75,50)(32.375,33.75)(16.5,18.5)
\qbezier(83.25,50)(99.875,33.75)(84,18.5)
\qbezier(52.75,53.25)(36.25,34.125)(50.75,14.5)
\qbezier(120.25,53.25)(103.75,34.125)(118.25,14.5)
\put(61,34.5){\vector(-1,0){50.5}}
\put(128.5,34.5){\vector(-1,0){50.5}}
%\qbezvec(48.25,21)(41.25,34.375)(49.25,46.25)
\put(49.25,46.25){\vector(2,3){.07}}\qbezier(48.25,21)(41.25,34.375)(49.25,46.25)
%\end
%\qbezvec(18.25,21.75)(28.375,34.25)(18,45.75)
\put(18,45.75){\vector(-3,4){.07}}\qbezier(18.25,21.75)(28.375,34.25)(18,45.75)
%\end
%\qbezvec(114.75,17.5)(105.625,34.625)(116,49.25)
\put(116,49.25){\vector(2,3){.07}}\qbezier(114.75,17.5)(105.625,34.625)(116,49.25)
%\end
%\qbezvec(88,20.75)(98.625,34.5)(86.75,48.25)
\put(86.75,48.25){\vector(-3,4){.07}}\qbezier(88,20.75)(98.625,34.5)(86.75,48.25)
%\end
\put(88.5,17.25){\makebox(0,0)[cc]{$U_k$}}
\put(114,14.25){\makebox(0,0)[cc]{$U$}}
\put(119,48){\makebox(0,0)[cc]{$\hat{a}$}}
\put(93,45.25){\makebox(0,0)[cc]{$\tau_k^{-s}$}}
\put(111,45.25){\makebox(0,0)[cc]{$\tau_a^r$}}
\put(101.5,32){\makebox(0,0)[cc]{$\hat{c}$}}
\put(34.25,32){\makebox(0,0)[cc]{$\hat{c}$}}
\put(19.25,40){\makebox(0,0)[cc]{$\tau_a^r$}}
\put(51,40.75){\makebox(0,0)[cc]{$\tau_0^{-s}$}}
\put(51.25,18){\makebox(0,0)[cc]{$U_0$}}
\put(16.25,21.75){\makebox(0,0)[cc]{$U$}}
\put(20.25,48.75){\makebox(0,0)[cc]{$\hat{a}$}}
\put(47,49.5){\makebox(0,0)[cc]{$\hat{b}$}}
\put(64.25,34.5){\makebox(0,0)[cc]{$X$}}
\put(7.75,34.5){\makebox(0,0)[cc]{$Y$}}
\put(75.75,34.75){\makebox(0,0)[cc]{$Y$}}
\put(131.5,34.75){\makebox(0,0)[cc]{$X$}}
\put(36,3.25){\makebox(0,0)[cc]{(a)}}
\put(104.25,3){\makebox(0,0)[cc]{(b)}}
\put(70.25,0){\makebox(0,0)[cc]{Figure 1}}
\put(86.25,43.75){\makebox(0,0)[cc]{$\hat{b}_k$}}
\put(81.5,16.75){\makebox(0,0)[cc]{$A$}}
\put(120,12.5){\makebox(0,0)[cc]{$E$}}
\put(81,51.75){\makebox(0,0)[cc]{$B$}}
\put(122.5,55){\makebox(0,0)[cc]{$F$}}
\end{picture}

\medskip
\medskip

 By the construction of $\tau_{\hat{a}}$, $g(\hat{b})\subseteq U$. Thus  $g(\mathbf{D}\backslash U_0)\subset U$. We conclude that for $k\geq 2$,
\begin{equation}\label{IMP}
g^k(U_0)\cap U\neq \emptyset,\  
g^k(U_0)\cup U=\mathbf{D}, \ \mbox{and}\ \partial \left(g^k(U_0)\right)\cap \{\hat{a}\}= \emptyset.
\end{equation}
The case of $k=1$ is interesting. If this occurs, then either (i) $g(U_0)\cap U\neq \emptyset$ and $g(U_0)\cup U=\mathbf{D}$; or (ii)
 $\{g(U_0),U\}$ tessellates the hyperbolic disk $\mathbf{D}$. The later occurs if and only if $\tilde{c}$ and $\tilde{a}$ intersect exactly once. See Section 5 for more details. 
 
For simplicity, we write $U_k=g^{k}(U_0)$, $U_k'=g^{k}(\mathbf{D}\backslash U_0)$, and $\hat{b}_k=g^{k}(\hat{b})$.
Figure 1 (b) depicts the situation that after $g^{k}$ for $k\geq 2$ is performed, the boundary geodesics $\hat{a}$ and $\hat{b}_k$ are disjoint and the union $U\cup U_k$ covers $\mathbf{D}$. 

Write $\mathscr{D}_k=U\cap U_k$ and $\mathbf{D}\backslash \mathscr{D}_k=\{\mathscr{R}_k, \mathscr{L}_k\}$, where $\mathscr{L}_k$ and $\mathscr{R}_k$ be the half planes covering the attracting and repelling fixed point of $g$, respectively. Clearly, the axis $\hat{c}$ of $g$ goes through the region $\mathscr{D}_k$. In what follows the hyperbolic length 
\begin{equation}\label{UI}
\varepsilon_k=\hat{c}\cap \mathscr{D}_k
\end{equation} 
is called the width of $\mathscr{D}_k$ with respect to $\hat{c}$. It is evident for $k=1$, $\varepsilon_1$ is less than the translation length $T_g$ of $g$ which is defined by 
$$
T_g=\mbox{inf} \left\{  \rho\left(z,g(z)\right):z\in \mathbf{D}\right\}.
$$
Furthermore, $\varepsilon_1=0$ if and only if $g(U_0)=\mathbf{D}\backslash \bar{U}$. Let $\gamma\subset S$ be a simple closed geodesic so that $\tilde{\gamma}\subset \tilde{S}$ is non-trivial. The positive Dehn twist $t_{\tilde{\gamma}}$ is well defined. By Lemma 3.2 of \cite{CZ1}, we can choose a lift $\tau_{\hat{\gamma}}:\mathbf{D}\rightarrow \mathbf{D}$ of $t_{\tilde{\gamma}}$ so that $[\tau_{\hat{\gamma}}]^*=t_{\gamma}$, where $\hat{\gamma}\subset \mathbf{D}$ is a geodesic with $\varrho(\hat{\gamma})=\tilde{\gamma}$. As usual, let $\mathscr{U}_{\hat{\gamma}}$ be the collection of half planes determined by $\tau_{\hat{\gamma}}$ and let $\Omega_{\hat{\gamma}}$ be the complement of all maximal elements of $\mathscr{U}_{\hat{\gamma}}$ in $\mathbf{D}$. 

Let $W\in \mathscr{U}_{\hat{\gamma}}$ be a maximal element. Then $\varrho(\partial W)=\tilde{\gamma}$. Assume that
$\partial W$ intersects $\hat{c}$ and $\tilde{\gamma}$ is disjoint from or equal to $\tilde{a}$. Denote by 
$$
W'=\mathbf{D}\backslash \bar{W}.
$$
\begin{lem}\label{L1}
With the above notation, assume also that $\varepsilon_k>T_g$. Then there is an element $W_0\in \mathscr{U}_{\hat{\gamma}}$ such that $W_0'\subset U$ or $W_0'\subset U_k$. 
\end{lem}
\begin{proof}
If $\partial W\subset \mathscr{D}_k$, then obviously, either $W'\subset U$ or $W'\subset U_k$. We are done with the choice $W_0=W$. 
If  $\partial W\subset \mathscr{R}_k$ and $W'\subset \mathscr{R}_k$, then $W'\subset U_k$. 
If $\partial W \subset \mathscr{R}_k$ and $W\subset \mathscr{R}_k$, then by Lemma 2.1, $g(\mathbf{D}\backslash W')$ is contained in an element $W_0$ of $\mathscr{U}_{\hat{\gamma}}$. Since $\varepsilon_k>T_g$, we have $W_0'\subset U_k$. 

It remains to consider the case of $\partial W\subset \mathscr{L}_k$. If $W'\subset \mathscr{L}_k$, then $W'\subset U$. If $W\subset \mathscr{L}_k$, then by Lemma 2.1 again, $g^{-1}(\mathbf{D}\backslash W')$ is contained in a maximal element $W_0$ of $\mathscr{U}_{\gamma}$. Since $\varepsilon_k>T_g$, we have $W_0'\subset U$. The lemma is proved. 
\end{proof}
\noindent {\em Remark. } If we know that $\partial W\neq \partial U$ or $\partial U_0$, then the conclusion of the lemma remains valid even if $\varepsilon_k=T_g$. This remark is useful in the proof of Theorem 1.2. 

\medskip
 
Define 
\begin{equation}\label{MAP}
\tau_k=g^{k}\tau_{\hat{a}}g^{-k}. 
\end{equation}
Denote by $\mathscr{U}_k$ the collection of half-planes determined by $\tau_k$, and by $\Omega_k$ the complement of the union of all maximal elements of $\mathscr{U}_k$. Then $\mathscr{U}_k=g^k(\mathscr{U}_{\hat{a}})$ and $\Omega_k=g^k(\Omega_{\hat{a}})$. It is clear that 
$\Omega_k\cap \Omega_{\hat{a}}=\emptyset$. As mentioned earlier, the equivalence classes $[\tau_k]$ and $[\tau_{\hat{a}}]$ are elements of QC$(G)/\! \sim$, and if a geodesic $\hat{a}$ is chosen properly, $[\tau_{\hat{a}}]^*$ is represented by $t_a$. By Lemma 3.2 of \cite{CZ1} again, $[\tau_k]^*$ is represented by the Dehn twist $t_{b_k}$ for some simple closed geodesic $b_k\subset S$. For simplicity,  we write $t_k=t_{b_k}$. It follows that 
\begin{equation}\label{EQUA}
[\tau_{\hat{a}}^r\tau_k^{-s}]^*= [\tau_{\hat{a}}^r]^*\circ  [\tau_k^{-s}]^* = t_{a}^r\circ t_k^{-s}.
\end{equation}
\noindent As usual, we let $\tilde{b}_k$ denote the geodesic on $\tilde{S}$ homotopic to $b_k$ on $\tilde{S}$. Then it is easy to show that $\tilde{b}_k=\tilde{b}$. 

\begin{lem}\label{L2}
Assume that $\varepsilon_k>T_g$ and there is a maximal element $W\in \mathscr{U}_{\hat{\gamma}}$ such that $\partial W$ intersects $\hat{c}$. Then as a mapping class of $Mod_S^x$, $[\tau_{\hat{a}}^r\tau_k^{-s}]^*$ has the property that $[\tau_{\hat{a}}^r\tau_k^{-s}]^*(\gamma)$ is not homotopic to $\gamma$. 
\end{lem}

\begin{proof}
Suppose that 
\begin{equation}\label{REDU}
[\tau_{\hat{a}}^r\tau_k^{-s}]^*(\gamma)=t_{a}^r\circ t_{k}^{-s}(\gamma)=\gamma.
\end{equation}
By Lemma 4.1 of \cite{CZ2}, $\tau_{\hat{a}}^r\tau_{k}^{-s}$  sends every maximal element of $\mathscr{U}_{\hat{\gamma}}$ to a maximal element of $\mathscr{U}_{\hat{\gamma}}$. This tells us that 
\begin{equation}\label{REDU0}
\tau_{\hat{a}}^r\tau_k^{-s}(W)=W_1, \ \ \mbox{where} \ W, W_1 \ \mbox{are maximal elements of}\  \mathscr{U}_{\hat{\gamma}}.
\end{equation}
From (\ref{REDU}), $t_{\tilde{a}}^r\circ t_{\tilde{a}}^{-s}(\tilde{ \gamma})=\tilde{\gamma}$. We have $t_{\tilde{a}}^{r-s}(\tilde{ \gamma})=\tilde{\gamma}$ for $r\neq s$.  So $\tilde{\gamma}$ must be disjoint from $\tilde{a}$. It follows that the sets $\{\varrho^{-1}(\tilde{\gamma})\}$ and $\{\varrho^{-1}(\tilde{a})\}$ are disjoint. We conclude that all boundary geodesics of maximal elements of $\mathscr{U}_{\hat{\gamma}}$ must be disjoint from $\hat{a}$ and $\hat{b}_k$. 

By hypothesis, $\partial W$ intersects $\hat{c}$. Thus Lemma \ref{L1} says that there is a maximal element $W_0\in \mathscr{U}_{\hat{\gamma}}$ such that $W_0'\subset U$ or $W_0'\subset U_k$. 

If $W_0'\subset U_k$, then since $\tilde{\gamma}$ is simple, the region  $\tau_{k}^{-1}(W_0')\subset U_k$ is disjoint from $W_0'$. So $\tau_{k}^{-s}(W_0')$ is near to the point $B$ ($s$ does not depend on $W_0'$ since $W_0'$ is away from the point $A$). Thus the region $\tau_{\hat{a}}^r\tau_k^{-s}(W_0')\subset U$ is near to the point $F$. Hence one of the following conditions holds (with $\zeta=\tau_{\hat{a}}^{r}\tau_{k}^{-s}$): 

(i) either  $\zeta(W_0')\subset W_0'$ and $W_0'\neq \zeta(W_0')$, or\\
\indent (ii) $\zeta(W_0')$ and $W_0'$ are disjoint.

\noindent In both cases, we have  
$W_0\cap  \tau_{\hat{a}}^{r}\tau_{k}^{-s}(W_0)\neq \emptyset$ and $W_0\neq \tau_{\hat{a}}^{r}\tau_{k}^{-s}(W_0)$. So 
$\tau_{\hat{a}}^{r}\tau_{k}^{-s}(W_0)$ is not a maximal element of $\mathscr{U}_{\hat{\gamma}}$.  This contradicts (\ref{REDU0}).

If $W_0'\subset U$, we consider the inverse map $\tau_{k}^{s}\tau_{\hat{a}}^{-r}$ of $\tau_{\hat{a}}^{r}\tau_{k}^{-s}$. Observe that $\tau_{\hat{a}}^{-r}(W_0')\subset U$ is near to the point $E$, and thus  $\tau_{k}^{s}\tau_{\hat{a}}^{-r}(W_0')\subset U_k$ is near to the point $A$. This implies that (i) and (ii) above remain valid with $\zeta=\tau_{k}^{s}\tau_{\hat{a}}^{-r}$. 
It follows that $\tau_{k}^{s}\tau_{\hat{a}}^{-r}(W_0)$ is not a maximal element of $\mathscr{U}_{\hat{\gamma}}$. This also contradicts (\ref{REDU0}). 
\end{proof}
\begin{lem}\label{L3}
Assume that $k\geq 3$. Then $\varepsilon_k>T_g$ and for all sufficiently large integers $r,s$,  $[\tau_{\hat{a}}^r\tau_k^{-s}]^*$ are pseudo-Anosov mapping classes. 
\end{lem}
\begin{proof}
Suppose that $[\tau_{\hat{a}}^r\tau_k^{-s}]^*$ is not pseudo-Anosov for some $k\geq 3$ and some large positive integers $r$ and $s$. By Bers \cite{Bers2}, it is either elliptic or reducible. From (\ref{EQUA}), $[\tau_{\hat{a}}^r\tau_k^{-s}]^*$ cannot be elliptic. Hence $[\tau_{\hat{a}}^r\tau_k^{-s}]^*$ must be a reducible mapping class. 

Let $\mathscr{C}=\{c_1,\ldots, c_{q}\}$, $q\geq 1$, be the corresponding curve simplex reduced by a representative of $[\tau_{\hat{a}}^r\tau_k^{-s}]^*$ (the curve simplex depends on $r,s$ and $k$). If $\mathscr{C}$ contains a curve $\gamma$ that bounds a twice punctured disk enclosing the puncture $x$, then such a $\gamma$ is unique in $\mathscr{C}$ (since any two twice punctured disks $\Delta, \Delta'$ must intersect if both $\Delta$ and $\Delta'$ enclose $x$).  From Lemma 5.1 and Lemma 5.2 of \cite{CZ1},  the restriction 
$\tau_{\hat{a}}^r\tau_{k}^{-s}\left|_{\partial \mathbf{D}}\right.$ fixes some parabolic fixed point of $G$. But Lemma 3.3 of \cite{CZ1} says that for sufficiently large integers $r$ and $s$, this does not occur. 

Suppose that $\mathscr{C}$ does not contain any curve $\gamma$ which is the boundary of a twice punctured disk enclosing $x$. This means that  $\mathscr{C}$ does not contain any curve  $\gamma$ with $\tilde{\gamma}$ being trivial (where $\tilde{\gamma}$ denotes the geodesic homotopic to $\gamma$ on $\tilde{S}$ ). Then we may assume (by taking a suitable power if necessary) that for a $\gamma\in \mathscr{C}$ with $\tilde{\gamma}$ being non-trivial, we have $t_{a}^r\circ t_{k}^{-s}(\gamma)=\gamma$. By Lemma \ref{L2}, $\hat{c}$ does not cross any element of $\mathscr{U}_{\hat{\gamma}}$. So either $\hat{c}\subset \Omega_{\hat{a}}$ or $\hat{c}\subset W$ for a maximal element $W$ of $\mathscr{U}_{\hat{\gamma}}$. 

If $\hat{c}\subset \Omega_{\hat{a}}$, $g$ commutes with $\tau_{\hat{\gamma}}$. So $f$ commutes with $t_{\gamma}$. This is impossible. Alternatively, the condition $\hat{c}\subset \Omega_{\hat{a}}$ leads to that $\varrho(\hat{c})$ is disjoint from $\tilde{a}$, contradicting that $\varrho(\hat{c})$ is a filling curve. 

It remains to consider the case that $\hat{c}\subset W$ for some $W\in \mathscr{U}_{\hat{\gamma}}$. In the following discussion we denote by $(XY)$ the minor arc in $\mathbf{S}$ connecting two points $X$ and $Y$ on $\mathbf{S}$.  
Notice that the boundary geodesic $\partial W$ of $W$ is disjoint from $\hat{a}$ and $\hat{b}_k$. We see that $\partial W$
is disjoint from $\hat{c}, \hat{a}$ and $\hat{b}_k$. Hence the arc $W'\cap \mathbf{S}^1$, where $W'=\mathbf{D}\backslash \bar{W}$, is a subarc of one of the six arc components of $\mathbf{S}^1-\{A,B,E,F,X,Y\}$. See Figure 1 (b) for these labeling points. 

Consider the case that $W'\cap \mathbf{S}^1\subset (AE)$. Since $\hat{c}\subset W$, $W'$ is disjoint from $U_k'$, $U'$ and $\hat{c}$. Note also that the Euclidean diameter of $\tau_{k}^{s}\tau_{\hat{a}}^{-r}(W')$ is strictly smaller than that of $W'$. Here $r$ and $s$ depend on $U$, $U_k$, $\hat{c}$, and does not depend on a particular half-plane $H$ with $H\cap \mathbf{S}^1\subset (AE)$ (since $H\cap \mathbf{S}^1$ is away from the points $B$ and $F$). 

We see that $\tau_{k}^{s}\tau_{\hat{\alpha}}^{-r}(W')\neq W'$. From Lemma 4.3 of \cite{CZ2}, $\tau_{k}^{s}\tau_{\hat{a}}^{-r}(W)\cap W\neq \emptyset$. Clearly, $W\neq \tau_{k}^{s}\tau_{\hat{a}}^{-r}(W)$. It follows that $\tau_{k}^{s}\tau_{\hat{a}}^{-r}$ and hence $\tau_{\hat{a}}^{r}\tau_{k}^{-s}$ does not send $W$ to any element of $\mathscr{U}_{\hat{\gamma}}$. If $W'\cap \mathbf{S}^1\subset (BF)$, then the same argument as above shows that $\tau_{\hat{a}}^{r}\tau_{k}^{-s}(W)$ is not an element of $\mathscr{U}_{\hat{\gamma}}$. 

When $W'\cap \mathbf{S}^1\subset (BY)\cup (YA)$, then since $\partial W$ projects to a simple closed geodesic on $\tilde{S}$, $\tau_{k}^{-s}\tau_{\hat{a}}^{r}(W')$ is disjoint from $W'$. Hence $\tau_{k}^{-s}\tau_{\hat{a}}^{r}(W)$ is not an element of $\mathscr{U}_{\hat{\gamma}}$. Similar discussion also applies to the case when $W'\cap \mathbf{S}^1\subset (EX)\cup (XF)$.  Details are omitted. 
\end{proof}

\section{Proof of Theorem 1.1}
\setcounter{equation}{0}

We first assume that $\hat{c}\cap \Omega_{\hat{a}}\neq \emptyset$. In this case, by Lemma \ref{L3}, the mapping classes $[\tau_{\hat{a}}^r\tau_k^{-s}]^*\in \mbox{Mod}_S^x$ are pseudo-Anosov when $k\geq 3$ and $r,s$ are sufficiently large positive integers with $r\neq s$ ($r,s$ depend on $k$).

Now we consider the case that $\hat{c}\cap \Omega_{\hat{a}}=\emptyset$. This means that $\hat{c}$ lies in a maximal element $U$ of $\mathscr{U}_{\hat{a}}$.  Denote by $\hat{a}=\partial U$. Let $Y$ and $X$ be the attracting and repelling  fixed points of $g$. See Figure 2. Since $\tilde{a}$ is a simple closed geodesic on $\tilde{S}$, for $k\geq 1$, $g^k(U')\cap U'=\emptyset$. More precisely, $g^k(U')\cap \mathbf{S}^1$ is a subarc contained in $(YE)$. 

In Figure 2, $g^k(U')$ is shown as the region $U_k'$. Let $U_k=\mathbf{D}\backslash U_k'$. Then  $g^k(U)=U_k$. We see that $U_k\cap U\neq \emptyset$, $U_k\cup U=\mathbf{D}$ and $\hat{b}_k \cap \hat{a}=\emptyset$, where $\hat{b}_k=\partial U_k$.  Following the notation introduced in Lemma \ref{L3},  we can similarly define $\tau_k$ as in (\ref{MAP}) and claim that for all $k\geq 1$, all mapping classes $[\tau_{\hat{a}}^r\tau_k^{-s}]^*$ are pseudo-Anosov when $r$ and $s$ are sufficiently large. Indeed, if $\hat{c}$ is contained in an element $W$ of $\mathscr{U}_{\hat{\gamma}}$, we can use the same argument as in the proof of Lemma \ref{L3} to conclude that this does not occur. 

We now assume that $\hat{c}$ crosses a maximal element $W$ of $\mathscr{U}_{\hat{\gamma}}$, as shown in Figure 2. Notice that $\tilde{\gamma}=\varrho(\hat{\gamma})$ is disjoint from $\tilde{a}$ and that $\partial U_k$ is a geodesic in $\{\varrho^{-1}(\tilde{a})\}$. We see that $\partial W$ must be  disjoint from $\hat{a}$ and $\partial U_k$. By Lemma \ref{A1}, $g(\mathbf{D}\backslash W)\subset W_0$, where $W_0\in \mathscr{U}_{\hat{\gamma}}$ is another maximal element disjoint from $W$. 

\bigskip
\medskip

%TeXCAD Picture [figure B.pic]. Options:
%\grade{\on}
%\emlines{\off}
%\epic{\off}
%\beziermacro{\on}
%\reduce{\on}
%\snapping{\off}
%\pvinsert{% Your \input, \def, etc. here}
%\quality{8.000}
%\graddiff{0.005}
%\snapasp{1}
%\zoom{4.0000}
\unitlength 1mm % = 2.845pt
\linethickness{0.4pt}
\ifx\plotpoint\undefined\newsavebox{\plotpoint}\fi % GNUPLOT compatibility
\begin{picture}(90,70)(0,0)
%\circle(64.5,43.25){50.658}
\put(89.829,43.25){\line(0,1){1.0848}}
\put(89.806,44.335){\line(0,1){1.0828}}
\put(89.736,45.418){\line(0,1){1.0789}}
\multiput(89.62,46.497)(-.032419,.21458){5}{\line(0,1){.21458}}
\multiput(89.458,47.569)(-.0297,.152139){7}{\line(0,1){.152139}}
\multiput(89.25,48.634)(-.031665,.131886){8}{\line(0,1){.131886}}
\multiput(88.997,49.689)(-.033142,.115919){9}{\line(0,1){.115919}}
\multiput(88.699,50.733)(-.031153,.093595){11}{\line(0,1){.093595}}
\multiput(88.356,51.762)(-.032205,.084493){12}{\line(0,1){.084493}}
\multiput(87.969,52.776)(-.0330413,.076649){13}{\line(0,1){.076649}}
\multiput(87.54,53.773)(-.0337014,.0697947){14}{\line(0,1){.0697947}}
\multiput(87.068,54.75)(-.0320772,.0597514){16}{\line(0,1){.0597514}}
\multiput(86.555,55.706)(-.0325712,.0548919){17}{\line(0,1){.0548919}}
\multiput(86.001,56.639)(-.0329538,.0504773){18}{\line(0,1){.0504773}}
\multiput(85.408,57.548)(-.0332389,.0464396){19}{\line(0,1){.0464396}}
\multiput(84.776,58.43)(-.0334375,.0427248){20}{\line(0,1){.0427248}}
\multiput(84.108,59.284)(-.0335588,.039289){21}{\line(0,1){.039289}}
\multiput(83.403,60.109)(-.0336102,.0360968){22}{\line(0,1){.0360968}}
\multiput(82.664,60.904)(-.0335982,.0331187){23}{\line(-1,0){.0335982}}
\multiput(81.891,61.665)(-.036576,.033088){22}{\line(-1,0){.036576}}
\multiput(81.086,62.393)(-.0397672,.0329907){21}{\line(-1,0){.0397672}}
\multiput(80.251,63.086)(-.0432009,.03282){20}{\line(-1,0){.0432009}}
\multiput(79.387,63.742)(-.0469125,.0325681){19}{\line(-1,0){.0469125}}
\multiput(78.496,64.361)(-.0509457,.032225){18}{\line(-1,0){.0509457}}
\multiput(77.579,64.941)(-.0553544,.031779){17}{\line(-1,0){.0553544}}
\multiput(76.638,65.482)(-.0642199,.0332962){15}{\line(-1,0){.0642199}}
\multiput(75.674,65.981)(-.0702718,.0326948){14}{\line(-1,0){.0702718}}
\multiput(74.69,66.439)(-.0771159,.0319363){13}{\line(-1,0){.0771159}}
\multiput(73.688,66.854)(-.084947,.030988){12}{\line(-1,0){.084947}}
\multiput(72.669,67.226)(-.103436,.032785){10}{\line(-1,0){.103436}}
\multiput(71.634,67.554)(-.116384,.031473){9}{\line(-1,0){.116384}}
\multiput(70.587,67.837)(-.132328,.029766){8}{\line(-1,0){.132328}}
\multiput(69.528,68.075)(-.177975,.032095){6}{\line(-1,0){.177975}}
\multiput(68.46,68.268)(-.215024,.029332){5}{\line(-1,0){.215024}}
\put(67.385,68.414){\line(-1,0){1.0804}}
\put(66.305,68.515){\line(-1,0){1.0837}}
\put(65.221,68.569){\line(-1,0){1.085}}
\put(64.136,68.576){\line(-1,0){1.0844}}
\put(63.052,68.538){\line(-1,0){1.0817}}
\multiput(61.97,68.452)(-.26927,-.03287){4}{\line(-1,0){.26927}}
\multiput(60.893,68.321)(-.17841,-.029583){6}{\line(-1,0){.17841}}
\multiput(59.822,68.143)(-.151696,-.031883){7}{\line(-1,0){.151696}}
\multiput(58.76,67.92)(-.131418,-.033557){8}{\line(-1,0){.131418}}
\multiput(57.709,67.652)(-.103888,-.031324){10}{\line(-1,0){.103888}}
\multiput(56.67,67.339)(-.093137,-.032495){11}{\line(-1,0){.093137}}
\multiput(55.646,66.981)(-.084022,-.033416){12}{\line(-1,0){.084022}}
\multiput(54.637,66.58)(-.0707258,-.0317009){14}{\line(-1,0){.0707258}}
\multiput(53.647,66.136)(-.064683,-.0323875){15}{\line(-1,0){.064683}}
\multiput(52.677,65.65)(-.0592842,-.0329326){16}{\line(-1,0){.0592842}}
\multiput(51.729,65.124)(-.0544182,-.0333567){17}{\line(-1,0){.0544182}}
\multiput(50.803,64.556)(-.0499985,-.0336759){18}{\line(-1,0){.0499985}}
\multiput(49.903,63.95)(-.0436593,-.0322077){20}{\line(-1,0){.0436593}}
\multiput(49.03,63.306)(-.0402284,-.0324267){21}{\line(-1,0){.0402284}}
\multiput(48.185,62.625)(-.0370389,-.032569){22}{\line(-1,0){.0370389}}
\multiput(47.371,61.909)(-.0340617,-.0326418){23}{\line(-1,0){.0340617}}
\multiput(46.587,61.158)(-.0326325,-.0340707){23}{\line(0,-1){.0340707}}
\multiput(45.837,60.374)(-.0325589,-.0370478){22}{\line(0,-1){.0370478}}
\multiput(45.12,59.559)(-.0324157,-.0402372){21}{\line(0,-1){.0402372}}
\multiput(44.44,58.714)(-.0321958,-.0436681){20}{\line(0,-1){.0436681}}
\multiput(43.796,57.841)(-.0336622,-.0500077){18}{\line(0,-1){.0500077}}
\multiput(43.19,56.941)(-.0333418,-.0544273){17}{\line(0,-1){.0544273}}
\multiput(42.623,56.015)(-.0329164,-.0592932){16}{\line(0,-1){.0592932}}
\multiput(42.096,55.067)(-.0323699,-.0646918){15}{\line(0,-1){.0646918}}
\multiput(41.611,54.096)(-.0316816,-.0707345){14}{\line(0,-1){.0707345}}
\multiput(41.167,53.106)(-.033393,-.084031){12}{\line(0,-1){.084031}}
\multiput(40.766,52.098)(-.03247,-.093146){11}{\line(0,-1){.093146}}
\multiput(40.409,51.073)(-.031296,-.103896){10}{\line(0,-1){.103896}}
\multiput(40.096,50.034)(-.033521,-.131427){8}{\line(0,-1){.131427}}
\multiput(39.828,48.983)(-.031842,-.151705){7}{\line(0,-1){.151705}}
\multiput(39.605,47.921)(-.029534,-.178418){6}{\line(0,-1){.178418}}
\multiput(39.428,46.85)(-.0328,-.26928){4}{\line(0,-1){.26928}}
\put(39.297,45.773){\line(0,-1){1.0817}}
\put(39.212,44.691){\line(0,-1){2.1694}}
\put(39.181,42.522){\line(0,-1){1.0837}}
\put(39.236,41.438){\line(0,-1){1.0804}}
\multiput(39.337,40.358)(.029391,-.215016){5}{\line(0,-1){.215016}}
\multiput(39.484,39.283)(.032144,-.177966){6}{\line(0,-1){.177966}}
\multiput(39.676,38.215)(.029803,-.13232){8}{\line(0,-1){.13232}}
\multiput(39.915,37.157)(.031504,-.116375){9}{\line(0,-1){.116375}}
\multiput(40.198,36.109)(.032814,-.103427){10}{\line(0,-1){.103427}}
\multiput(40.526,35.075)(.031011,-.084939){12}{\line(0,-1){.084939}}
\multiput(40.899,34.056)(.0319574,-.0771072){13}{\line(0,-1){.0771072}}
\multiput(41.314,33.053)(.032714,-.0702629){14}{\line(0,-1){.0702629}}
\multiput(41.772,32.07)(.0333138,-.0642108){15}{\line(0,-1){.0642108}}
\multiput(42.272,31.106)(.0317941,-.0553457){17}{\line(0,-1){.0553457}}
\multiput(42.812,30.165)(.0322389,-.0509369){18}{\line(0,-1){.0509369}}
\multiput(43.393,29.249)(.0325809,-.0469036){19}{\line(0,-1){.0469036}}
\multiput(44.012,28.357)(.0328318,-.0431919){20}{\line(0,-1){.0431919}}
\multiput(44.668,27.494)(.0330015,-.0397582){21}{\line(0,-1){.0397582}}
\multiput(45.361,26.659)(.033098,-.036567){22}{\line(0,-1){.036567}}
\multiput(46.089,25.854)(.0331279,-.0335891){23}{\line(0,-1){.0335891}}
\multiput(46.851,25.082)(.0361059,-.0336003){22}{\line(1,0){.0361059}}
\multiput(47.646,24.342)(.0392982,-.033548){21}{\line(1,0){.0392982}}
\multiput(48.471,23.638)(.0427339,-.0334258){20}{\line(1,0){.0427339}}
\multiput(49.326,22.969)(.0464487,-.0332262){19}{\line(1,0){.0464487}}
\multiput(50.208,22.338)(.0504863,-.03294){18}{\line(1,0){.0504863}}
\multiput(51.117,21.745)(.0549008,-.0325562){17}{\line(1,0){.0549008}}
\multiput(52.05,21.192)(.0597601,-.0320609){16}{\line(1,0){.0597601}}
\multiput(53.006,20.679)(.0698039,-.0336823){14}{\line(1,0){.0698039}}
\multiput(53.984,20.207)(.076658,-.0330203){13}{\line(1,0){.076658}}
\multiput(54.98,19.778)(.084502,-.032182){12}{\line(1,0){.084502}}
\multiput(55.994,19.392)(.093603,-.031128){11}{\line(1,0){.093603}}
\multiput(57.024,19.049)(.115928,-.03311){9}{\line(1,0){.115928}}
\multiput(58.067,18.751)(.131895,-.031629){8}{\line(1,0){.131895}}
\multiput(59.122,18.498)(.152147,-.029658){7}{\line(1,0){.152147}}
\multiput(60.187,18.291)(.214589,-.032361){5}{\line(1,0){.214589}}
\put(61.26,18.129){\line(1,0){1.0789}}
\put(62.339,18.013){\line(1,0){1.0829}}
\put(63.422,17.944){\line(1,0){2.1697}}
\put(65.592,17.944){\line(1,0){1.0828}}
\put(66.675,18.014){\line(1,0){1.0788}}
\multiput(67.753,18.131)(.214571,.032478){5}{\line(1,0){.214571}}
\multiput(68.826,18.293)(.152131,.029741){7}{\line(1,0){.152131}}
\multiput(69.891,18.501)(.131878,.031701){8}{\line(1,0){.131878}}
\multiput(70.946,18.755)(.11591,.033173){9}{\line(1,0){.11591}}
\multiput(71.989,19.053)(.093586,.031179){11}{\line(1,0){.093586}}
\multiput(73.019,19.396)(.084484,.032229){12}{\line(1,0){.084484}}
\multiput(74.033,19.783)(.07664,.0330622){13}{\line(1,0){.07664}}
\multiput(75.029,20.213)(.0697855,.0337204){14}{\line(1,0){.0697855}}
\multiput(76.006,20.685)(.0597426,.0320935){16}{\line(1,0){.0597426}}
\multiput(76.962,21.199)(.054883,.0325862){17}{\line(1,0){.054883}}
\multiput(77.895,21.753)(.0504683,.0329676){18}{\line(1,0){.0504683}}
\multiput(78.803,22.346)(.0464306,.0332516){19}{\line(1,0){.0464306}}
\multiput(79.685,22.978)(.0427156,.0334492){20}{\line(1,0){.0427156}}
\multiput(80.54,23.647)(.0392798,.0335695){21}{\line(1,0){.0392798}}
\multiput(81.365,24.352)(.0360876,.0336201){22}{\line(1,0){.0360876}}
\multiput(82.159,25.091)(.0331096,.0336072){23}{\line(0,1){.0336072}}
\multiput(82.92,25.864)(.033078,.0365851){22}{\line(0,1){.0365851}}
\multiput(83.648,26.669)(.0329798,.0397762){21}{\line(0,1){.0397762}}
\multiput(84.34,27.504)(.0328082,.0432098){20}{\line(0,1){.0432098}}
\multiput(84.997,28.369)(.0325553,.0469214){19}{\line(0,1){.0469214}}
\multiput(85.615,29.26)(.0322111,.0509545){18}{\line(0,1){.0509545}}
\multiput(86.195,30.177)(.0317638,.055363){17}{\line(0,1){.055363}}
\multiput(86.735,31.119)(.0332787,.064229){15}{\line(0,1){.064229}}
\multiput(87.234,32.082)(.0326756,.0702808){14}{\line(0,1){.0702808}}
\multiput(87.692,33.066)(.0319152,.0771247){13}{\line(0,1){.0771247}}
\multiput(88.106,34.068)(.030965,.084956){12}{\line(0,1){.084956}}
\multiput(88.478,35.088)(.032757,.103445){10}{\line(0,1){.103445}}
\multiput(88.806,36.122)(.031441,.116392){9}{\line(0,1){.116392}}
\multiput(89.089,37.17)(.02973,.132336){8}{\line(0,1){.132336}}
\multiput(89.326,38.229)(.032047,.177984){6}{\line(0,1){.177984}}
\multiput(89.519,39.297)(.029273,.215032){5}{\line(0,1){.215032}}
\put(89.665,40.372){\line(0,1){1.0804}}
\put(89.765,41.452){\line(0,1){1.7979}}
%\end
\qbezier(87,54.5)(70.625,43.875)(86.75,31.75)
\qbezier(41,34)(62.625,40.375)(64.75,18.25)
\qbezier(39.75,48.5)(59.75,46.875)(57.75,67.75)
\qbezier(48.75,63)(63.375,50.5)(72.5,67)
\put(83.75,48.75){\makebox(0,0)[cc]{$\hat{a}$}}
\put(85.75,35.75){\makebox(0,0)[cc]{$U'$}}
\put(83,30.75){\makebox(0,0)[cc]{$U$}}
\put(67.25,20.75){\makebox(0,0)[cc]{$U_k$}}
\put(61,21.25){\makebox(0,0)[cc]{$U_k'$}}
\put(48.5,32){\makebox(0,0)[cc]{$\hat{b}_k$}}
\put(52.75,47.75){\makebox(0,0)[cc]{$\hat{c}$}}
\put(67.5,65.25){\makebox(0,0)[cc]{$W$}}
\put(90,56){\makebox(0,0)[cc]{$F$}}
\put(89.5,30){\makebox(0,0)[cc]{$E$}}
\put(65,15){\makebox(0,0)[cc]{$A$}}
\put(38,33.25){\makebox(0,0)[cc]{$B$}}
\put(73.5,70){\makebox(0,0)[cc]{$Q$}}
\put(47,64.75){\makebox(0,0)[cc]{$P$}}
\put(57.5,70){\makebox(0,0)[cc]{$X$}}
\put(37,48.75){\makebox(0,0)[cc]{$Y$}}
\put(64.75,7.75){\makebox(0,0)[cc]{Figure 2}}
%\emline(44.5,49.25)(42.25,48.25)
\multiput(44.5,49.25)(-.075,-.0333333){30}{\line(-1,0){.075}}
%\end
%\emline(42.25,48.25)(44.75,47.25)
\multiput(42.25,48.25)(.0833333,-.0333333){30}{\line(1,0){.0833333}}
%\end
\end{picture}

Write $\partial W\cap \mathbf{S}^1=\{P,Q\}$. Observe that $g^k(U')=U_k'$, we have $g(E)=B$ and $g(F)=A$.  This implies the point $g(P)$ stays in the arc connecting $X$ and $Y$; while $g(Q)$ either lies in the arc connecting $Q$ and $F$, or the arc connecting $E$ and $A$ (the point $g(Q)$ cannot lie in the arc connecting $F$ and $E$. For otherwise, $g(\partial W)$ would intersect $\hat{a}$. This contradicts that $g(\partial W)$ projects to $\tilde{\gamma}$ which is disjoint from $\tilde{a}$). It follows from the same argument of Lemma \ref{L2} and Lemma \ref{L3}  that for $k\geq 1$, the mapping classes $[\tau_{\hat{a}}^r\tau_k^{-s}]^*\in \mbox{Mod}_S^x$ are pseudo-Anosov whenever $r$ and $s$ are sufficiently large positive integers with $r\neq s$. 

We conclude that no matter whether $\hat{c}\cap \Omega_{\hat{a}}\neq \emptyset$ or not, for $k\geq 3$,  $[\tau_{\hat{a}}^r\tau_k^{-s}]^*\in \mbox{Mod}_S^x$ are pseudo-Anosov when $r$ and $s$ are sufficiently large positive integers with $r\neq s$.  From (\ref{EQUA}), this is equivalent to that $t_{a}^r\circ t_k^{-s}$ represents a pseudo-Anosov mapping class on $S$. It follows from Thurston's theorem \cite{Th} that all $(a, b_k)$ for $k\geq 3$ fill $S$. But from (\ref{MAP}) we have 
\begin{equation}\label{CAL}
t_k=[\tau_k]^*=(g^*)^k\circ [\tau_{\hat{a}}]^*\circ (g^*)^{-k}=f^k\circ t_a\circ f^{-k}=t_{f^k(a)}.
\end{equation}
We conclude that $b_k=f^k(a)$. Hence $(a,f^k(a))$ fills $S$ for $k\geq 3$. This proves the first statement of Theorem 1.1.

To prove the second statement of the result, we notice that if $\tilde{c}$ intersects $\tilde{a}$ more than once, then the translation length $T_g$ is larger than the length of the segment of $\hat{c}$ in $\mathbf{D}\backslash \left( U\cup U_0 \right)$. See Figure 1 (a). This means that $\varepsilon_2>T_g$. By Lemma 3.3, the mapping classes $[\tau_{\hat{a}}^r\tau_{2}^{-s}]^*$ are pseudo-Anosov for all large integers $r$ and $s$ with $r\neq s$, i.e., $t_a^r\circ t_2^{-s}$ are pseudo-Anosov. It follows that $(a,b_2)$ fills $S$. But a simple calculation similar to (\ref{CAL}) reveals that $b_2=f^2(a)$. As a consequence, $(a,f^2(a))$ fills $S$. Now by a similar argument, $(a,f^m(a))$ also fills $S$ for all $m\geq 2$.  

The last statement of the theorem is the restatement of Theorem 1.1 of \cite{CZ3}. \ \ \ \ \ \ \ \ \ \ \ \ \ \ \ \ \ \ \ \ \ \ \ \ \ \ \ \ \ \ \ \ \ \ \ \ \ \ \ \ \ $\Box$

\section{Proof of Theorem 1.2 and Corollary 1.1}
\setcounter{equation}{0} 
We continue to write $f=g^*$, where we recall that $g\in G$ is an essential hyperbolic element with axis $\hat{c}$.  Let $\varepsilon_k$ be defined as in (\ref{UI}). We refer to Figure 1 (a). 
\begin{lem}\label{L51}
$\varepsilon_2\geq T_g$, and $\varepsilon_2= T_g$ if and only if $\varepsilon_1=0$. 
\end{lem}
\begin{proof}
By the definition, $\varepsilon_2$ is the width of $\mathscr{D}_2=U\cup g^2(U_0)$. We see that  $\varepsilon_2=\varepsilon_1+T_g\geq T_g$, and $\varepsilon_2= T_g$ if and only if $\varepsilon_1=0$. 
\end{proof}
\begin{lem}\label{L52}
$0\leq \varepsilon_1< T_g$, and $\varepsilon_1=0$ if and only if $a$ and $f(a)$ are disjoint and furthermore $\{a,f(a)\}$ forms the boundary of an $x$-punctured cylinder on $S$.  
\end{lem}
\begin{proof}
It is trivial that $\varepsilon_1\geq 0$. Since $\Omega_{\hat{a}}$ has a non-empty interior, $\varepsilon_1<T_g$. If $\varepsilon_1=0$, then $\{g(U_0),U\}$ tessellates the hyperbolic plane $\mathbf{D}$. By (\ref{MAP}), $\tau_1=g\tau_{\hat{a}}g^{-1}$.
By Lemma 3.2 of \cite{CZ1}, there is a geodesic $\alpha\subset S$ such that 
\begin{equation}\label{J1}
[\tau_1]^*=t_{\alpha}.
\end{equation} 
Let $h\in G$ be the primitive hyperbolic element that keeps $\hat{a}=\partial U$ invariant and takes the same orientation as that of $\tau_{\hat{a}}$. By inspecting the actions of $\tau_{\hat{a}}$ and $\tau_1$ on $\mathbf{D}$, we find that $\tau_{\hat{a}}\tau_1^{-1}$ coincides with $h$ off the geodesic $\hat{a}$. Hence $\tau_{\hat{a}}\tau_1^{-1}|_{\partial \mathbf{D}}=h|_{\partial \mathbf{D}}$. This tells us that 
\begin{equation}\label{J2}
[\tau_{\hat{a}}\tau_1^{-1}]^*=h^*.
\end{equation} 
From Theorem 2 of \cite{Kr} and Theorem 2 of \cite{Na}, $h^*=t_a\circ t_{a_0}^{-1}$, where $\{a,a_0\}$ is the boundary of an $x$-punctured cylinder $P$ on $S$. It follows
from (\ref{J1}) and (\ref{J2}) that 
$$
t_a\circ t_{\alpha}^{-1}=[\tau_{\hat{a}}\tau_1^{-1}]^*=h^*
=t_a\circ t_{a_0}^{-1}. 
$$
So $t_{\alpha}=t_{a_0}$ and hence $\alpha=a_0$. A calculation similar to (\ref{CAL}) yields
$a_0=\alpha=g^*(a)=f(a)$. This proves that $(a,f(a))$ forms the boundary of $P$. 

Conversely, if $\varepsilon_1>0$, then $\mathscr{D}_1\neq \emptyset$. By Lemma 4 of \cite{CZ0}, $\tau_{\hat{a}}$ and $\tau_1$ do not commute. Thus $t_a$ and $t_{\alpha}$ do not commute, which is equivalent to that $a$ and $f(a)$ intersect. In particular, $\{a,f(a)\}$ does not form a boundary of any $x$-punctured cylinder on $S$. 
\end{proof}

\noindent {\em Proof of Theorem $1.2$: }  By assumption, $(a,f^2(a))$ does not fill $S$. By Theorem 1.1 of \cite{CZ3}, $a$ cannot be the boundary component of a twice punctured disk enclosing $x$. That is, $\tilde{a}$ is non-trivial. Let $\tau_{\hat{a}}$ be the lift of $t_{\tilde{a}}$ so that $[\tau_{\hat{a}}]^*=t_a$. From the argument of Theorem 1.1, we also claim that 
$\hat{c}\cap \Omega_{\hat{a}}\neq \emptyset$. We are now in the situation depicted in Figure 1 (a). By Lemma \ref{L51}, $\varepsilon_2\geq T_g$. 

If $\varepsilon_2> T_g$, then by Lemma \ref{L3}, for large $r,s$ with $r\neq s$, the mapping classes $[\tau_{\hat{a}}^r\tau_{2}^{-s}]^*$ are pseudo-Anosov, i.e., $t_a^r\circ t_{2}^{-s}$ are pseudo-Anosov. This implies that $(a,b_2)$ fills $S$. But a calculation similar to (\ref{CAL}) demonstrates that $b_2=f^2(a)$. As a result, $(a,f^2(a))$ fills $S$. This is a contradiction. 

We thus conclude that $\varepsilon_2 = T_g$. Thus $\varepsilon_1 = 0$. Therefore, by Lemma \ref{L52}, $a$ and $f(a)$ are disjoint and $\{a, f(a)\}$ forms the boundary of an $x$-punctured cylinder $P$ on $S$. Note that $f:S\rightarrow S$ is a homeomorphism, $f(a)$ and $f^2(a)$ are also disjoint. We see that $f(a)$ is disjoint from both $a$ and $f^2(a)$. 

We need to show that if $\gamma\subset S$ is a geodesic disjoint from both $a$ and $f^2(a)$, then $\gamma=f(a)$. For this purpose, we notice that $t_{\gamma}$ commutes with $t_a^r\circ t_{f^2(a)}^{-s}=t_a^r\circ \left(f^2\circ t_a^{-s}\circ f^{-2}\right)$. By the Bers isomorphism, $\tau_{\hat{\gamma}}$ commutes with $\tau_{\hat{a}}^r\left(g^2\tau_{\hat{a}}^{-s}g^{-2}\right)$. By Lemma 4.1 of \cite{CZ2}, $\tau_{\hat{a}}^r\left(g^2\tau_{\hat{a}}^{-s}g^{-2}\right)$ sends every maximal element $W$ of $\mathscr{U}_{\hat{\gamma}}$ to a maximal element. 

First we claim that $g(U_0)=\mathbf{D}\cup U$. See Figure 1 (a). Suppose $g(U_0)\cap U\neq \emptyset$. Then $g(U_0)\cup U=\mathbf{D}$ and 
$\varepsilon_2>T_g$.  By the same argument of Lemma \ref{L3}, $\left[\tau_{\hat{a}}^r\left(g^2\tau_{\hat{a}}^{-s}g^{-2}\right)\right]^*$ are pseudo-Anosov. This contradicts that $\tau_{\hat{a}}^r\left(g^2\tau_{\hat{a}}^{-s}g^{-2}\right)$ sends every maximal element $W\in \mathscr{U}_{\hat{\gamma}}$ to a maximal element.

Next we claim that there exists a maximal element $W\in \mathscr{U}_{\hat{\gamma}}$ such that 
\begin{equation}\label{U1}
W=g(U_0).
\end{equation}
 Since $\tilde{c}$ is a filling geodesic, if there is no maximal element $W$ with $\partial W\cap \hat{c}\neq \emptyset$, there is a maximal element $W_0\in \mathscr{U}_{\hat{\gamma}}$ such that $\hat{c}\subset W_0$. Since $\partial W_0=\partial W_0'$ does bot intersect $\hat{a}$, $\hat{b}_k$ and $\hat{c}$, we see that $W_0'\cap \mathbf{S}^1$ lies in one of the six arcs of $\mathbf{S}^1\backslash \{A,B,E,F,X,Y\}$ (see Figure 1 (b)). The same argument of Lemma \ref{L3} yields that this is impossible. 

Consider the case when there is a maximal element $W\in \mathscr{U}_{\hat{\gamma}}$  with $\partial W\cap \hat{c}\neq \emptyset$. Notice that $\varepsilon_1=0$
and $\varepsilon_2=T_g$. If $W\neq g(U_0)$, the same argument of Lemma 3.1 can be used to assert that there is a maximal element $W_0\in  \mathscr{U}_{\hat{\gamma}}$ such that either $W_0'\subset U$ or $W_0'\subset U_2$. Then we can use the argument of Lemma 3.2 to deduce that $\tau_{\hat{a}}^r\tau_{2}^{-s}(W_0)$ is not a maximal element of $\mathscr{U}_{\hat{\gamma}}$. This contradicts that $\tau_{\hat{a}}^r\left(g^2\tau_{\hat{a}}^{-s}g^{-2}\right)$ sends every maximal element $W\in \mathscr{U}_{\hat{\gamma}}$ to a maximal element. 

We conclude that (\ref{U1}) holds. Since $U_0$ is a maximal element of $\mathscr{U}_{\hat{a}}$, (\ref{U1}) tells us that $\tau_{\hat{\gamma}}=g\tau_{\hat{a}}g^{-1}$. With the help of the Bers isomorphism, we obtain $t_{\gamma}=f\circ t_a\circ f^{-1}=t_{f(a)}$. That is, $\gamma=f(a)$. This completes the proof of Theorem 1.2. \ \ \ \ \ \ \ \ \ \ \ \ \ \ \ \ \ \ \ \ \ \ \ \ \ \ \ \ \ \ \ \ \ \ \ \ \ \ \ \ \   $\Box$

\medskip

\noindent {\em Proof of Corollary $1.1$: } Suppose that $(a,f^{k_0}(a))$ does not fill $S$ for some integer $k_0\geq 2$. This means that there is a simple closed geodesic $b\subset S$ such that $b$ is disjoint from both $a$ and $f^{k_0}(a)$.  By Theorem 1.2, $b=f(a)$. This implies that $a$ and $f(a)$ are disjoint, which leads to a contradiction. \ \ \ \ \ \ \ \ \ \ \ \ \ \ \ \ \ \ \ \ \ \ \ \ \ \ \ \ \ \ \ \ \ \ \ \ \ \ \ \ \ \ \ \ \ \ \ \ \ \ \ \ \ \ \ \ \ \ \ \ \ \ \ \ \ \ \ \ \ \ \ \ \ \ \ \ \ \ \ \ \ \ \ \ \ \ \ \   $\Box$ 
 
\section{Remarks}
\setcounter{equation}{0}

For any simple closed geodesic $a\subset S$, if $\tilde{a}\subset \tilde{S}$ is trivial, then Theorem 1.1 of \cite{CZ3} states that $(a,f(a))$ fills $S$ for all elements $f\in \mathscr{F}$. We know that the set $\mathscr{F}$ is a disjoint union of $\mathscr{F}_{\tilde{c}}$, where every $\mathscr{F}_{\tilde{c}}$ consists of pseudo-Anosov maps conjugate (in the fundamental group $\pi_1(\tilde{S},x)$) to a map $f_{\tilde{c}}$ that can be obtained from a filling closed geodesic $\tilde{c}\subset \tilde{S}$. From the argument of Theorem 1.1, for every filling closed geodesic $\tilde{c}$, most elements $f$ in $\mathscr{F}_{\tilde{c}}$ have the property that $(a,f(a))$ fill $S$. However, there are some exceptions. 

From Theorem 1.2, some filling closed geodesics $\tilde{c}\subset \tilde{S}$ (and thus elements $f\in \mathscr{F}_{\tilde{c}}$) can be identified so that $(a,f^2(a))$ does not fill $S$ (while Theorem 1.1 guarantees that $(a,f^3(a))$ always fills $S$). A question asks whether some filling geodesic $\tilde{c}$ (and thus elements $f\in \mathscr{F}_{\tilde{c}}$) can be identified so that $(a,f(a))$ does not fill $S$ but $(a,f^2(a))$ does. If this occurs, then $\tilde{c}$ must intersect $\tilde{a}$ in a complicated way.    

\medskip

\noindent {\bf Acknowledgment. } I am grateful to the referees for pointing out to me some recent developments in the subject, and for their helpful comments and thoughtful suggestions.

\end{document}